\documentclass[letterpaper, 10 pt, conference]{ieeeconf}  

\IEEEoverridecommandlockouts                              

\usepackage{url}
\usepackage{array,tabularx}
\usepackage{multicol, nonfloat}  
\usepackage{epsfig} 
\usepackage{amssymb}  
\newtheorem{theorem}{Theorem}[section]
\newtheorem{corollary}[theorem]{Corollary}

\newtheorem{remark}{Remark}

\usepackage{blindtext}
\usepackage[utf8]{inputenc}	
\usepackage[english]{babel}
\usepackage{graphicx}
\usepackage{amsmath}
\usepackage{amsfonts}
\usepackage{epstopdf}
\usepackage{cite}
\usepackage{microtype}

\newcommand\numberthis{\addtocounter{equation}{1}\tag{\theequation}}

\title{\LARGE \bf
On a Potential Game-theoretic Approach to Event-triggered Distributed
Resource Allocation}

\author{Prashant Bansode, Sharad Jadhav, Mukesh Patil, Navdeep Singh
\thanks{Prashant Bansode, Sharad Jadhav are with Department of Instrumentation Engineering, Mukesh Patil is with Department of Electronics and Telecommunication Engineering, Ramrao Adik Institute of Technology, Mumbai, 400706 India.~{\tt\small prashant.bansode@rait.ac.in}}
\thanks{Navdeep Singh is with department of Electrical Engineering, Veermata Jijabai Technological Institute, Mumbai, 400019 India.}
}
\begin{document}

\maketitle
\thispagestyle{empty}
\pagestyle{empty}

\begin{abstract}
This paper proposes a potential game theoretic approach to address event-triggered distributed resource allocation in multi-agent systems. The fitness dynamic of the population is proposed and exploited as a linear parameter-varying dynamic to achieve the event-triggered distributed resource allocation. Firstly, it is shown that the dynamic resembles a consensus protocol and is equivalent to the replicator dynamic. Further, conditions ensuring a minimum time bound on the inter-sampling intervals to avoid Zeno-behavior have been derived. An economic dispatch problem is presented to illustrate the theoretical results. 
\end{abstract}

\section{Introduction}
Population games have been widely studied with interest to resource allocation in multi-agent systems. Recently, there has been a rising count in such applications \cite{ marden2009cooperative,ramirez2010population,pantoja2012distributed,pantoja2011population,pantoja2011dispatch,barreiro2014constrained,mojica2014dynamic,obando2014building,7172156,barreiro2016distributed}. One can see that population games find immense interest in managing public utilities while taking conservative measures to reduce unnecessary wastage. A population game theoretic approach to the cooperative control problems in the multi-agent systems is reported in \cite{ marden2009cooperative}. Application of such games to water distribution system, luminance control of lighting zones, optimal dispatch of distributed generators and building temperature control is presented in \cite{ramirez2010population,pantoja2012distributed,pantoja2011population,pantoja2011dispatch,obando2014building}, respectively. The literature contribution shows that the distributed replicator dynamic is a widely explored protocol for modeling the evolution of the resources in distributed resource allocation problems.

A general approach to event-triggered, real-time scheduler for stabilizing control tasks has been presented in \cite{tabuada2007event}. The associated issue of achieving asymptotic synchronization in multi-agent systems under event-triggered protocols has been widely addressed since then \cite{liuzza2013distributed,demir2012event,seyboth2013event,dimarogonas2012distributed,nowzari2014zeno,selivanov2016event,dimarogonas2009event,lemmon2010event,deshpande2017distributed}. Event triggered control strategies are intended to address the lack of availability of shared-band communication network reduce communication cost. A recent work in the application of the event-triggered distributed optimization technique to the multi-agent systems is considered in \cite{5211941,7505245,7500070}. It is well known that the event-triggered distributed optimization and control techniques are susceptible to Zeno-behavior\cite{lemmon2010event}. Zeno is a behavior in event-triggered systems where subsequent triggering instances lie in such a close proximity that, they pose a threat to stability of the associated system.. Hence, eliminating such behavior is an important aspect in event-trigger based control and optimization techniques. 

Although popular, the existing literature of population games with application to distributed resource allocation applications viz., \cite{ ramirez2010population,pantoja2012distributed,pantoja2011population,pantoja2011dispatch,barreiro2014constrained,mojica2014dynamic,obando2014building, 7172156,barreiro2016distributed,doi:10.1080/00207179.2016.1231422} does not explore the likelihood of event-triggered distributed resource allocation. Since there already exists wide interest in application of population games to distributed resource allocation, it motivates investigating the applicability of the event-triggered control techniques to the underlying problem.
This paper specifically deals with a framework for integrating  population game theoretic approaches with the event-triggered control framework in order to achieve distributed resource allocation in the multi-agent systems. The reported work envelopes following contributions:
\begin{enumerate}
	\item To develop fitness consensus protocol for multi-agent systems.
	\item To integrate the fitness consensus protocol with distributed event-triggered control framework and enable its implementation over shared communication channels.
	\item To ensure asymptotic synchronization in presence of event-triggered interruptions.
	\item To ensure Zeno-free operation of the event-triggered, distributed resource allocation.
\end{enumerate}
The rest of the paper is organized as follows: Some preliminaries on graph theory, potential games, and distributed resource allocation in multi-agent systems are presented in Section \ref{sec2}. Section \ref{sec3} presents the distributed replicator dynamic and the proposed fitness consensus protocol. Section \ref{sec4} discusses event-triggered control framework to the fitness consensus protocol. 
Subsection \ref{sim} illustrates theoretical results with an economic dispatch problem in an islanded microgrid. Section \ref{sec5} concludes the paper.
\section{Preliminaries}\label{sec2}
\subsection{Graph theory}\label{sec2.a}
A communication graph $\mathcal{G}$, with $n$ nodes is defined as a tuple $\mathcal{G}=(\mathcal{N}, \mathcal{E})$, with nodes $\mathcal{N}= \{1, 2, \ldots, n\}$ and edges $\mathcal{E} \subseteq \mathcal{N} \times \mathcal{N}$. Let the neighbor set of the $i^{\mathrm{th}}$ agent be defined as $\mathcal{N}_i=\{j\in \mathcal{N}| (j,i) \in \mathcal{E} \}$, where $i\in\{1,2,\ldots,n\}$. The number $n$ is the cardinality of the graph $n=|\mathcal{N}|$.
\subsection{Population games}\label{sec2.b}
Assuming that there exists a fixed set of strategically interacting populations, a population game is described by its payoff function\cite{sandholm2010population}. The agents of a given population choose their strategies from a pure strategy set $\mathcal{H}=\{\mathcal{H}_1,\mathcal{H}_2,\ldots,\mathcal{H}_n\}$ and their aggregate behavior is described as a population state $p=\{p_1,\ldots,p_n\}$ that belongs to a simplex $\bigtriangleup$ for all $t\geq 0$ such that
\begin{equation}
\bigtriangleup=\{p(t) \in \Re^n_+:\sum_{i=1}^{n}p_i(t)=P_{\mathrm{tot}}\}\label{simplex}
\end{equation}\cite{sandholm2010population}. 
The total population is given by $\sum_{i=1}^{n}p_i=P_{\mathrm{tot}}$. A payoff is a continuous vector-valued function $f=\{f_1,\dots,f_n\}:\mathcal{R}^n_+ \rightarrow \mathcal{R}^n$. Payoff function is also regarded as the fitness function of the population. The scalar $p_i$ denotes the subset of the total population choosing the pure strategy $\mathcal{H}_i$ whose associated fitness is $f_i$. If there exists a continuously differentiable, strictly quadratic concave potential function $S(p):\mathcal{R}^n_+ \rightarrow \mathcal{R}$, i.e. $S(p)=-p^T \Pi p + b^Tp +c$ with the matrix $\Pi>0$ such that $f(p)=\nabla S(p)$, then a population game is referred to as a potential game.\vspace{-0.3cm}
\subsection{Distributed Resource Allocation}\label{dra}
Changing the view-point, let $p_i$ now denote the number of resources to be allocated to the $i^{\mathrm{th}}$ agent in the network of a multi-agent system. Here $P_{\mathrm{tot}}$ denotes the total resources available for the purpose of utilization. A utility maximization problem is considered where the global utility is the summation of the utilities associated with all agents. The distributed resource allocation problem takes the following form:
\begin{equation}
\setlength\arraycolsep{1.5pt}
\begin{array}{r c}
\mathrm{maximize} ~U(p),\\
\mathrm{subject~to}~ p \in \bigtriangleup, 
\end{array} \label{optip}
\end{equation} 
where $U(p)=\sum_{i=1}^{n}u_i(p_i)$. $u_i(p_i)$ is a real-valued utility function associated with the $i^{\mathrm{th}}$ agent in the network and $U(p)$ is the global utility function. It is assumed that $u_i(p_i)$ is strictly concave-quadratic, implying the global utility function is also strictly concave in $p$.
\section{Replicator dynamic}\label{sec3}
A replicator dynamic is a revision protocol that describes choice-based strategy revisions of the agents. This dynamic is now applied to study the strategic interactions in a multi-agent system. Let $p_i$ denote the subset of the total population associated with the $i^{\mathrm{th}}$ agent then its dynamic is described by
\begin{equation}
\dot{p}_i=p_i[f_i(p)-\bar{f}(p)], \label{cp}
\end{equation}
with the average fitness of the population computed as
\begin{equation*}
\bar{f}(p)=\frac{1}{P_{\mathrm{tot}}}\sum_{j=1}^{n}p_jf_j(p).\label{af}
\end{equation*}
The distributed version of the replicator dynamic \eqref{cp} is given by \cite{bravo2015distributed,barreiro2016distributed}
\begin{equation}
\dot{p}_i=\sum_{j\in \mathcal{N}_i}^{}p_ip_j(f_i(p)-f_j(p)).\label{lre}
\end{equation}
Further, it is well known that the simplex \eqref{simplex}
remains invariant under dynamics \eqref{cp} and \eqref{lre}\cite{barreiro2016distributed}. 
\subsection{Multi-Agent Distributed Replicator Dynamic}
The multi-agent system dynamics can be represented by
\begin{eqnarray}
\dot{p}(t)&=&L(p)f(p), \label{nw}
\end{eqnarray}
where $p(t)$ is the stack vector of agents' states and $L(p)\in \mathcal{R}^{n \times n}$ is the Laplacian matrix of the interaction graph $\mathcal{G}=\{\mathcal{N},{\mathcal{E}}\}$ that represents the multi-agent system, i.e.
\begin{equation}
L(p) \triangleq \begin{cases}
{l}_{i,j}=-p_ip_j, \forall (i,j) \in \mathcal{E},\\
{l}_{i,j}=0, \forall (i,j) \not\in \mathcal{E},\\
{l}_{i,i}=-\sum_{j\in \mathcal{N}_i}l_{i,j}. 
\end{cases}\label{aii}
\end{equation}
Let $\mathcal{N}=\{f_1,f_2,\ldots,f_n\}$ represent the set of fitness perceived by the agents which opt strategies $\mathcal{H}=\{\mathcal{H}_1,\mathcal{H}_2,\ldots,\mathcal{H}_n\}$, respectively and $\mathcal{\mathcal{E}}$ be the set of edges in $\mathcal{G}$.

In view of its definition \eqref{aii}, the graph Laplacian matrix $L(p)$ is a parameter varying, real and symmetric positive semidefinite matrix which is differentiable and uniformly continuous in $p$. As a consequence the following hold:
\begin{enumerate}
	\item  There exists $\eta>0$ such that the spectral norm $||L(p)||<\eta, p \in \bigtriangleup$.
	\item The gradient of $L(p)$ with respect to $p$ is bounded as given by
	$||\nabla L(p)||\leq \xi, p \in \bigtriangleup$, where $\zeta>0$.
\end{enumerate}
\subsection{Fitness Consensus Protocol}
If $f_i(p)$ is the fitness associated with population $p_i$ then the following relationship holds:
\begin{equation}
f_i(p)=\frac{\partial S(p)}{\partial p_i}. \label{partial}
\end{equation}
Considering equation \eqref{partial}, a stack vector of agents' fitness is given by
\begin{equation}
{f}(p)=\nabla_p S(p). \label{nabla}
\end{equation} 
Differentiating \eqref{nabla} with respect to time, the fitness dynamic of the multi-agent system \eqref{nw} is given by:
\begin{align*}
\dot{f}(p)=&\nabla(\nabla S(p))\dot{p}, \\
=& -\Pi \dot{p}, \numberthis \label{8b}\\
=&-N(p)f(p),\numberthis \label{aut}
\end{align*}
where $N(p)=\Pi L(p)$ which exhibits the following properties:
\begin{align}
N(p)\geq 0,\label{pr1}\\
N(p)\boldsymbol{1_n} = 0, \label{pr2}
\end{align} where $\boldsymbol{1_n}$ is a vector of all ones
(see \cite{wu1988products} for \eqref{pr1}). The properties \eqref{pr1} and \eqref{pr2} are sufficient to show that the fitness dynamic \eqref{aut} resembles an agreement protocol\cite{mesbahi2010graph}.
%
\section{Event-Triggered Distributed Resource Allocation}\label{sec4}
This section considers a distributed resource allocation problem described in Section \ref{dra}.
A related point to consider here is that the global utility function $U(p)$ in \eqref{optip} is equal to the potential function $S(p)$ introduced in Section \ref{sec2.b} when the underlying problem is allowed to be solved under the framework of distributed replicator dynamic as described in\cite{ pantoja2012distributed,pantoja2011population,pantoja2011dispatch,barreiro2014constrained,mojica2014dynamic,obando2014building, 7172156,barreiro2016distributed}. Interestingly, the Hessian matrix of utility function \eqref{optip} results in a diagonal matrix.

Before proceeding with the event-triggered, distributed resource allocation problem, the following observation must be noted.
\begin{remark}\label{thm3.1}
	For the distributed resource allocation problem defined in equation \eqref{optip}, the resulting dynamic systems \eqref{nw} and \eqref{aut} follow a linear relationship and their steady states are concurrent. It is because of the diagonal structure of the Hessian matrix $-\Pi$ that the derivatives $\dot{f}_i(p)$ and $\dot{p}_i$ form a linear relationship as presented in the following equation:
	\begin{equation}
	\dot{f}_i(p)=-\pi_{ii}\dot{p}_i, \label{relat}
	\end{equation}
	where $\pi_{ii}$ is $i^{\mathrm{th}}$ diagonal constant in the matrix $\Pi$. 
	The linear relationship between these two dynamics also contributes to concurrency in their steady states.
\end{remark}
\subsection{Event Triggered Control Approach to Fitness Consensus Problem}
\indent This subsection describes a sampling-based fitness consensus protocol wherein the agents utilize the communication channel only when the current fitness value is novel i.e., significantly different than the last fitness value communicated. Hence
\begin{eqnarray}
\dot{f}(p)=-N(p)\hat{f}(p),\label{sampled}
\end{eqnarray}
where $\hat{f}(p)=\{\hat{f}_1,\ldots,\hat{f}_n\}$ is related to $f(p)=\{f_1,\ldots,f_n\}$ through the gap function vector $e(f(p))=\hat{f}(p)-f(p)$. Equation \eqref{sampled} then modifies to
\begin{equation*}
\dot{f}(p)=-N(p)f(p)-N(p)e(f(p)).\label{samplede}
\end{equation*}
\begin{remark}
	As the dynamic \eqref{sampled} requires both population and fitness scalars to be communicated among the neighboring agents, under event-triggered sampling, it is required to sample both the scalars for every agent in the network. From equation \eqref{relat}, it is seen that the fitness function $f_i(p)$ is affine in the population $p_i$ associated with the $i^{\mathrm{th}}$ agent. The similar relationship holds for the sampled version of these scalars. Hence the sampled fitness function $\hat{f}_i(p)$ can be used to compute the respective sampled population scalar $\hat{p}_i$. Considering fitness as a state associated with the agent, under an event-trigger rule, both fitness and population scalars are communicated between the neighboring agents, whenever the gap function exceeds a certain threshold say $e_T$, i.e. $|e(f(p))|=|\hat{f}(p(t_k))-f(p(t))|>e_T$, where $t_k$ is the $k^{\mathrm{th}}$ consecutive sampling time. It guarantees both event-triggered distributed resource allocation and event-triggered fitness consensus. \hfill$\blacksquare$
\end{remark}
The next subsection provides the stability analysis of the sampled dynamic \eqref{sampled}.
\subsection{Stability Analysis of Fitness Consensus Problem}
The stability analysis of the dynamic \eqref{sampled} pertains to the following two possibilities associated with the connectedness of the graph $\mathcal{G}$:
\subsubsection{Graph $\mathcal{G}$ is connected $\forall t \geq 0$}\label{sec4c1}
\begin{theorem}\label{thm2}
	Consider the system \eqref{sampled} with the candidate Lyapunov function given below
	\begin{equation}
	V(f) = \frac{1}{2}f^T(p)\Pi^{-1}f(p)\label{v}.
	\end{equation}
	Then dynamic \eqref{sampled} is globally asymptotically stable if for some $\rho=\{\rho_1,\ldots,\rho_n\colon0\le \rho_i\le 1,\forall i=1,\ldots,n\}$ and $a>0$, the following holds:
	\begin{equation*}
	e^2_i \leq \frac{\rho_i a (\lambda_2-a|\mathcal{N}_i|)}{|\mathcal{N}_i|}f^2_i.\label{impose1}
	\end{equation*}
\end{theorem}
\begin{proof} Differentiating \eqref{v} along the trajectories of the system \eqref{sampled} yields
	\begin{align*}
	\dot{V}(f)&=f^T(p)\Pi^{-1}\dot{f}(p),\label{fdo}\\
	&=-f^T(p)L(p)f(p)-f^T(p)L(p)e(f(p)).
	\end{align*}
	Let $\lambda_2$ be the second smallest eigenvalue of the graph Laplacian matrix $L(P)$. Since the graph $\mathcal{G}$ remains connected $\forall t \geq 0$, the following property holds:
	\begin{equation*}
	f^T(p)L(p)f(p)\geq \lambda_2f^T(p)f(p),
	\end{equation*}
	which leads to the following.
	\begin{align*}
	\dot{V}(f)&\leq-\lambda_2f^T(p)f(p)-f^T(p)L(p)e(f(p)),\\
	&\leq-\lambda_2\sum_{i}^{n}f^2_i-\sum_{i}^{n}\sum_{j \in \mathcal{N}_i}^{}f_i(e_i-e_j),\\
	&\leq-\lambda_2\sum_{i}^{n}f^2_i-\sum_{i}^{n}|\mathcal{N}_i|f_ie_i+\sum_{i}^{n}\sum_{j \in \mathcal{N}_i}^{}f_ie_j.
	\end{align*}
	Using Young's inequality and acknowledging the fact that the graph $\mathcal{G}$ is symmetric \cite{dimarogonas2009event} yields
	\begin{eqnarray}
	\dot{V}(f)&\leq&-\sum_{i}^{n}(\lambda_2-a|\mathcal{N}_i|)f^2_i+\sum_{i}^{n}\frac{1}{a}|N_i|e^2_i. \label{vdot}
	\end{eqnarray}
	With $a$ chosen such that $(\lambda_2-a|\mathcal{N}_i|)>0, i \in \mathcal{N}$, the following distributed event-trigger rule is derived
	\begin{equation}
	e^2_i \leq \frac{\rho_i a (\lambda_2-a|\mathcal{N}_i|)}{|\mathcal{N}_i|}f^2_i.\label{impose}
	\end{equation}
	Equation \eqref{vdot} is further modified as
	\begin{equation*}
	\dot{V}(f)\leq\sum_{i=1}^{n} (\rho_i-1)	(\lambda_2-a|\mathcal{N}_i|)f^2_i.\label{vdot1}
	\end{equation*}
	Choosing $0<\rho_i<1$ yields $\dot{V}(f)\leq0$. Hence proved. 
\end{proof}
In addition to that, using LaSalle's invariance principle, the next corollary proves that $\dot{V}(f)=0$ only at the convergence value, i.e. $\bar{f}(p)$.
\begin{corollary}\label{cor1}
	For the dynamic \eqref{sampled}, the derivative $\dot{V}(f)=0$ only when all agents converge to the agreement value $\bar{f}(p)$.
\end{corollary}
\begin{proof} Consider first the relation $\dot{V}(f)=f^T(p)\Pi^{-1}\dot{f}(p)$. The consensus protocol \eqref{sampled} achieves a steady state when $\lim_{t\rightarrow \infty}f_i(p)=\bar{f}(p),\forall i\in \mathcal{N}$. Hence for $0<\rho_i<1$, $\dot{V}(f)=0$ only if the agent dynamic is at steady state, i.e. $\dot{f}_i(p)=0,\forall i =1,\ldots,n$.
\end{proof}
\subsubsection{The graph $\mathcal{G}$ has connected components}
Let there be $c$ connected components in $\mathcal{G}$, where $c$ is the algebraic multiplicity of eigenvalue $\lambda_1(L(p))=0$. Each connected component, i.e. $\mathcal{G}_m,~\forall m = 1,\ldots,c$, has an associated Laplacian $L_m(p)$. In such a case, the matrix $L(p)$ has a block diagonal form:
\begin{equation*}
L(p)=\left(\begin{array}{ccccc}
L_1(p) & \hfill & \hfill & \hfill & \hfill\\
\hfill & L_2(p) & \hfill & \hfill & \hfill\\
\hfill & \hfill &  \ddots & \hfill & \hfill\\
\hfill & \hfill & \hfill & L_{c-1}(p) & \hfill\\
\hfill & \hfill & \hfill & \hfill & L_{c}(p)\\
\end{array}\right).
\end{equation*} 
\indent It is obvious that each connected component $\mathcal{G}_m$ has an eigenvalue $\lambda_1(L_m(p))=0$ with algebraic multiplicity $1$. In line with this, the asymptotic stability analysis from Subsection \ref{sec4c1} extends to the connected components of $\mathcal{G}$.

\indent In what follows, the minimum bound on inter-sampling interval is derived.
\begin{theorem}\label{thm3}
	The system \eqref{sampled} avoids Zeno-behavior if the inter-sampling intervals $t_{k+1}-t_{k}, \forall k=(1,\ldots,\infty)$ are lower bounded by a strictly positive time $\tau$ given by
	$\tau = \frac{1}{||\Pi||\eta+||\Pi||\eta\phi}.$
\end{theorem}
\begin{proof} Let the centralized ISS event-trigger rule be of the form $||e(f(p))||^2\leq \gamma||f(p)||^2$ with some $\gamma>0$, then the following relationship between $||e(f(p))||$ and $||f(p)||$ holds,
	\begin{equation}
	\frac{||e(f(p))||}{||f(p)||}\leq \sqrt{\gamma}, ~\text{for} ~t \in (t_{k+1},t_{k}). \label{gam}
	\end{equation}
	The next triggering occurs when condition \eqref{gam} is violated, i.e. $\frac{||e(p(t_{k+1}))||}{||f(p(t_{k+1}))||}> \sqrt{\gamma}$.
	Taking time derivative of the ratio $\frac{||e(f(p))||}{||f(p)||}$ as discussed in \cite{tabuada2007event} yields,
	\begin{align*}
	\frac{d}{dt}\frac{||e||}{||f||}&=\frac{||f||\frac{d}{dt}||e||-||e||\frac{d}{dt}||f||}{||f||^2},\\
	&=\Bigg(1+\frac{||e||}{||f||}\Bigg)\frac{||\dot{f}||}{||{f}||},\\
	&=||N(p)||\Bigg(1+\frac{||e||}{||f||}\Bigg)^2. \numberthis \label{26c}
	\end{align*}
	From the condition $||L(p)||<\eta$, $||N(p)||$ can be realized as $||\Pi||\eta$.	Further, let $x=\frac{||e||}{||f||}$, equation \eqref{26c} simplifies to
	\begin{eqnarray}
	\dot{x}&\leq&||\Pi||\eta(1+x)^2. \label{diffe}
	\end{eqnarray}
	The solution of \eqref{diffe} is given by 
	$x(\tau,0)=\frac{||\Pi||\eta\tau}{1-||\Pi||\eta\tau}.$
	If the $k^{th}$ gap function $e(p(t_{k}))$ violates the event-trigger condition \eqref{gam} at time instant $t_{k+1}$, then 
	$\frac{||e(f(p))||}{||f(p)||}> \sqrt{\gamma}.$
	Let $\phi$ be a positive constant such that
	\begin{equation*}
	\phi ||e(p(t_{k+1}))|| \geq (\sqrt{\gamma})^{-1}||e(p(t_{k+1}))|| \geq ||f(P(t_{k+1}))||.
	\end{equation*}
	Then the next sample occurs when 
	$\frac{||e(p(t_k))||}{||f(p(t))||}\geq \phi^{-1},$
	which leads the next event-trigger occurring when
	\begin{equation*}
	\frac{1-||\Pi||\eta\tau}{||\Pi||\eta\tau} \geq \frac{||e(p(t_{k+1}))||}{||f(p(t_{k+1}))||} \geq \phi^{-1}.
	\end{equation*}
	Solving for $\tau$ gives
	\begin{equation}
	\tau = \frac{1}{||\Pi||\eta+||\Pi||\eta\phi}.\label{diffe2}
	\end{equation}
	As long as the lower bound on $\tau$ in \eqref{diffe2} is greater than $0$, the event-triggered system \eqref{sampled} guarantees not to exhibit Zeno-behavior. The proof is complete.
\end{proof}

The following theorem is motivated from the results on minimum inter-sampling interval discussed for the case of the decentralized event-trigger rule in \cite{dimarogonas2009event}.
\begin{theorem}\label{thm4}
	Assuming that the rule \eqref{impose} applies to all $i \in \mathcal{N},~\forall t \geq 0$, then the $i^{th}$ agent dynamic exhibits Zeno-free behavior if its inter-sampling interval $t_{i,k+1}-t_{i,k}, \forall k=(1,\ldots,\infty)$ is lower bounded by a positive time constant $\tau^i$ such that
	$\tau^i=\frac{\beta_i}{||\Pi||\eta(\beta_i+q_i)},$
	for some $q_i>0$, and $\beta_i=\sqrt{\frac{\rho_i a (\lambda_2-a|\mathcal{N}_i|)}{|\mathcal{N}_i|}}$.
\end{theorem}
\begin{proof} It is obvious that $|e_i|\leq||e||$ and $|f_i|\leq ||F||$. Using these relationships, the following is obtained
	\begin{align*}
	\frac{|e_i|}{||e||}\leq 1,\frac{||F||}{|f_i|}\geq1,\\
	\frac{|e_i|}{||e||}\frac{||F||}{|f_i|}\leq q_i,\\
	\frac{|e_i|}{|f_i|}\leq q_i\frac{||e||}{||F||}. \numberthis \label{kida}
	\end{align*}
	Substituting for $\frac{||e||}{||F||}$ in \eqref{kida},
	\begin{align*}
	\frac{|e_i|}{|f_i|}\leq& q_i\frac{||\Pi||\eta\tau^i}{1-||\Pi||\eta\tau^i},\\
	\beta_i=&q_i\frac{||\Pi||\eta\tau^i}{1-||\Pi||\eta\tau^i}.
	\end{align*}
	Solving for $\tau^i$ yields
	$\tau^i=\frac{\beta_i}{||\Pi||\eta(\beta_i+q_i)}.$
	Hence proved.
\end{proof}
\subsection{Numerical Example and Simulation Results}\label{sim}
For the sake of illustration, an economic dispatch problem in an islanded microgrid is considered with a test system of four distributed generators and four loads. The representative system is shown in Fig. \ref{f0} as well as the cost parameters of the DGs are shown in Table \ref{table1}, see \cite{chen2016distributed}.
\begin{figure}[!t]
	\centering
	\includegraphics[width = 0.495\textwidth, keepaspectratio]{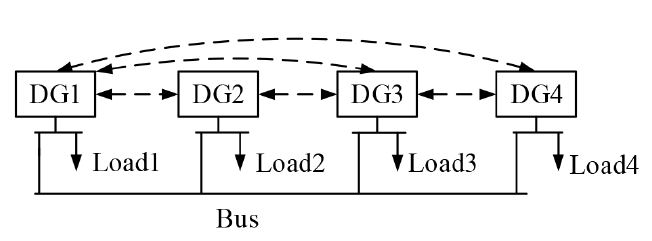}
	\caption{Structure of the microgrid.}
	\label{f0}
\end{figure}
The cost function for each generator is given by
\begin{equation}
C_i(P_i)=\alpha_iP^2_i+\beta_iP_i+\gamma_i. \label{dispatch}
\end{equation}

The objective of the economic dispatch problem is to minimize the operational cost of the microgrid while satisfying the demand-supply constraint. This problem is formulated as given below:
\begin{equation}
\setlength\arraycolsep{1.5pt}
\begin{array}{r c}
\mathrm{minimize}~~ \sum_{i=1}^{n} C_i(P_i), \\
\mathrm{subject~to}~\sum_{i=1}^{n}P_i=P_D. 
\end{array} \label{convex}
\end{equation} 

For the purpose of simplification, the lower and upper bounds on the power generation have been relaxed. The problem \eqref{convex} is now posed as a concave optimization problem and the event-triggered distributed resource allocation framework is applied as discussed in Section \ref{sec4}. A distributed event-triggered rule \eqref{impose} is considered with parameters $\rho$ and $a$ stated as
$\rho=[0.06, 0.01, 0.08, 0.05]^T$,
$a=e^{-6}$.
The power generation allocated to each DG in the microgrid is analogous to the amount of resources allocated or population assigned to each agent in the multi-agent system. 
When the active power of the loads are configured initially to $P_1 = 50KW$, $P_2 = 50KW$, $P_3 = 20KW$, and $P_4 = 20KW$, the optimal power generations converge to $P^*_1=19.2112KW$, $P^*_2=104.0474KW$, $P^*_3=11.3250KW$, and $P^*_4=5.4164KW$ as shown in Fig. \ref{f1}.a. The fitness function, which takes the form of an incremental cost of operating the DGs, converges to $f^*=-4.9085\$/KW$ as shown in \ref{f1}.b (note: the minus sign associated with the fitness function indicates the concavity of the objective function). Fig. \ref{f1}.c shows that the demand-supply constraint represented in equation \eqref{convex} is always satisfied. Fig. 2.d represents the event-trigger instances at which
the DGs communicated their incremental operational costs
and power generated to their neighboring DGs. It can be
observed that as the incremental costs from Fig. 2.b and power
generations from Fig. 2.a approach their respective steady state values, the communication between the neighboring DGs become more sparse. This communication eventually disappears as both scalars reach their respective steady state
values indicating that the gap function $e_T$ is $0$.
\begin{table}[!t]
	\centering
	\renewcommand{\arraystretch}{1.3}
	\caption{Generator cost parameters.}
	\label{table1}
	\centering
	\begin{tabular}{c|c|c|c|c}
		\hline
		{$DG_i$}&  {1} &{2} &{3} &{4}\\
		\hline\hline
		$\alpha_i$& 0.096& 0.072& 0.105&0.082 \\
		\hline
		$\beta_i$&1.22&3.41&2.53&4.02\\
		\hline
		$\gamma_i$&51&31&72&48\\
		\hline
	\end{tabular}
\end{table}
\begin{figure}[!t]
	\centering
	\includegraphics[width = 0.5\textwidth, keepaspectratio]{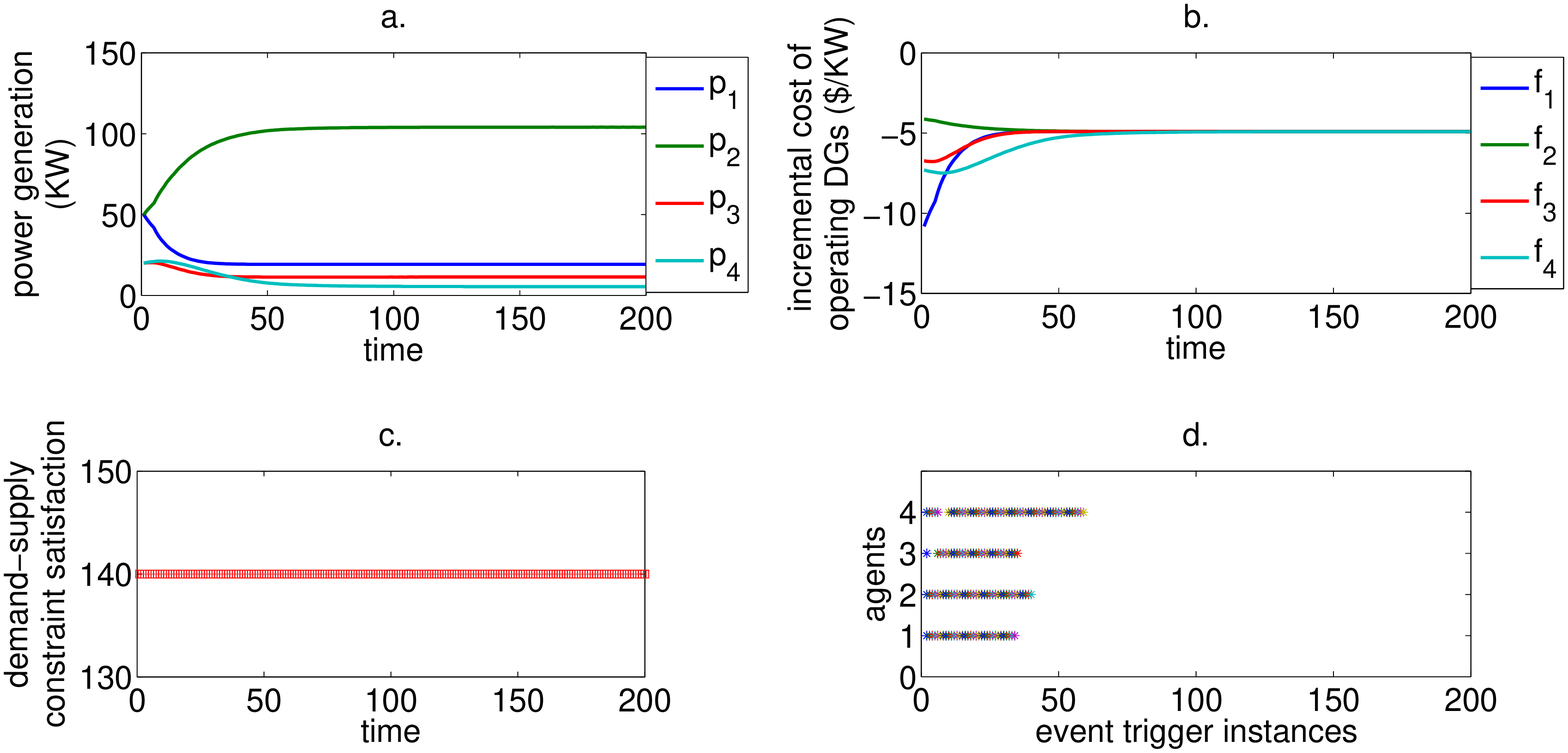}
	\caption{a. Power generation (population) b. incremental cost of operating DGs (fitness) c. demand-supply constraint satisfaction (population invariance) d. event-trigger instances.}
	\label{f1}
\end{figure}	
\section{CONCLUSIONS}\label{sec5}
In this paper, an application of concave potential games to event-triggered distributed resource allocation in the multi-agent systems is presented. The equivalence between two dynamics viz., distributed replicator dynamic and fitness dynamic is proved. A special attention is given to the fitness dynamic which reveals to be a dynamic consensus protocol. Further, this consensus protocol is effectively exploited to achieve the event-triggered distributed resource allocation. The equilibrium point of the consensus protocol is shown to be asymptotically stable under centralized and distributed event-trigger rules. A minimum inter-sampling interval essentially required to avoid Zeno-behavior is also derived. In the end, an economic dispatch problem in an islanded microgrid is presented to demonstrate the proposed approach.   	
\bibliographystyle{unsrt}  
\bibliography{trans} 
\end{document}